\documentclass[11pt]{article}
\usepackage{graphicx}
\usepackage{mathrsfs}
\usepackage{amsmath}
\usepackage{color}
\usepackage{bbm}
\usepackage{amsmath,amsthm,amssymb,amscd}
\usepackage{latexsym}
\usepackage{hyperref}
\usepackage[numbers,sort&compress]{natbib}
\usepackage{hypernat}

\newtheorem{theorem}{Theorem}[section]
\newtheorem{corollary}[theorem]{Corollary}
\newtheorem{lemma}[theorem]{Lemma}
\newtheorem{proposition}[theorem]{Proposition}

\theoremstyle{definition} \theoremstyle{remark}
\numberwithin{equation}{section}

\begin{document}\title{{\bf Dominance of slow solutions for second order abstract evolution equations with time-varying damping
}\thanks{The work was supported partly by the NSF of China (12171094), and the Shanghai Key
Laboratory for Contemporary Applied Mathematics (08DZ2271900).}}
\author{Jun-Ren Luo$^{a}$,  Ti-Jun Xiao$^{b}$ \thanks{Corresponding author. \ E-mail: \ tjxiao@fudan.edu.cn}
\\{\small $^a$  College of Science, University of Shanghai for Science and Technology, Shanghai 200093, China}\\
{\small $^b$  Shanghai Key
Laboratory for Contemporary Applied Mathematics}\\ {\small School of Mathematical Sciences, Fudan University, Shanghai 200433, China}}

\date{}
\maketitle

\begin{minipage}[2cm]{14cm}

\centerline {\bf Abstract}

   Of concern is a class of non-autonomous evolution equations of second order in Hilbert spaces, with a nonnegative self-adjoint operator $A$, time-varying damping and nonlinear source term.
 We give an upper decay rate of the energy, valid for all solutions and solely based on the damping coefficient and the geometrical index of the source term.
 Furthermore, we prove under suitable conditions that for all initial data, except for those in the kernel of $A$,
 the solutions decay (in the energy norm) at most as fast as this decay rate.
   The result not only shows the optimality of the decay rate, but also reveals an unusual phenomenon: ``slow solutions",
   i.e. those that decay at {\it exactly} this rate, are dominant in amount. Moreover, specialized to the case when the nonlinear source is absent, our result improves relevant existing ones to a large extent.

\vspace{0.4cm}


\vspace{0.24cm}


\end{minipage}

\section{Introduction}

Let $A$ be a nonnegative self-adjoint linear operator on a Hilbert space $\mathcal{H}$, and let the nonlinear operator $\nabla F: D\left(A^{1 / 2}\right) \rightarrow \mathcal{H}$ be the gradient of a scalar-valued function $F(u)$ having a polynomial growth comparable to $\|u\|^{s}$ for some $s>2$.
We consider the following abstract hyperbolic equation with initial data:
\begin{eqnarray}\label{1.1}
\left\{\begin{array}{ll}
u^{\prime \prime}(t)+A u(t)+\gamma(t) u^{\prime}(t)+\nabla F(u(t))=0, \quad \forall t \geq t_0>0, \\[0.36cm]
u(t_0)=u_{0},~ u^{\prime}(t_0)=u_{1},
\end{array} \right.
\end{eqnarray}
where the damping coefficient takes the form
\begin{equation}\label{gamma}
\gamma(t)=\frac{\alpha}{t},\quad \mbox{with some}~\alpha>0.
\end{equation}
The abstract model includes, as special cases, various concrete PDEs and integrodifferential systems of constant or variable coefficients.

Letting $F=0$ in the equation of \eqref{1.1} yields the linear equation $$u^{\prime \prime}(t)+A u(t)+\gamma(t) u^{\prime}(t)=0.$$ For this equation with $A=-\Delta$ in $L^2(R^n)$, Matsumura \cite{MAT}(see also \cite{USE}) obtained
\begin{equation}\label{Emu}
E(t)=O(t^{-\mu}), \,\,~ \mu:=\min\{2,\alpha\},
\end{equation}
by the aid of weighted energy inequalities. Later, Wirth \cite{WIR} gave an alternative proof of such estimates using an explicit representation of the solution operator by special functions, illustrating why this change in the decay order for $\mu = 2$ appears.
It is shown by Hirosawa and Nakazawa \cite{HIR} that for $\alpha>2$,
$$
\lim _{t \rightarrow+\infty}(1+t)^{2}E(t)=0.
$$
For $\alpha<1,$  Wirth \cite{WIR1} furthermore showed from a modified scattering theory that all nonzero solutions decay with the rate $t^{-\alpha}$. Recently,
assuming that $A$ is unitary equivalent to a non-negative multiplication operator in some $L^2$ space, and $m/t\le\gamma(t)\le M/t$ (with positive constants $M,m$),
Ghisi and Gobbino \cite{GHI0} proved \eqref{Emu} by Fourier analysis, and pointed out that if $A$ is coercive, then all solutions decay at least as fast as the rate $t^{-\alpha}$, even if $\alpha>2$. See also \cite{GHIG} and the references therein for related researches.

Letting $A=0$ in the equation of \eqref{1.1} gives
\begin{equation}\label{1.1-0}
  u^{\prime \prime}(t)+\gamma(t) u^{\prime}(t)+\nabla F(u(t))=0.
\end{equation}
For this equation with $F$ being convex,
  we know from \cite{AUJ} (see also \cite{LUO4}) that for \eqref{1.1-0}, the energy $E(t)$  satisfies
\begin{equation}\label{nA}
            E(t)=   \left\{
            \begin{array}{lll}
             O\left(t^{-\frac{2s}{s-2}}\right),  &{\rm if}~~ \alpha\ge\frac{s+2}{s-2},\\[0.28cm]
           O\left(t^{-\frac{2\alpha s}{s+2}}\right), &{\rm if}~~0<\alpha\le\frac{s+2}{s-2},
           \end{array}
        \right.
\end{equation}
and the decay rates are optimal. It is worth noting that $\frac{s+2}{s-2}$ is a {\it critical value} of $\alpha.$

On the other hand, for the full equation \eqref{1.1} in the presence of both operators $A$ and $F$, Alvarez and Attouch in \cite{ALVA} obtained the energy decay rate $t^{-1}$ when $\gamma\equiv\alpha$ and $F$ is convex; see also \cite{ATT0, CAB}. If $F(u)$ is comparable to $\|u\|^{s}$ ($s>2$), Ghisi, Gobbino, and Haraux \cite{GHI} obtained, in the case of $\gamma\equiv 1$, the optimal decay rate $t^{-s/(s-2)}$ (see also \cite{GHI1}). In the case of $\gamma(t)=\alpha / t^{\theta}$ (with some $\theta \in(0,1)$ and $\alpha>0$), the optimal energy decay rate $t^{-\frac{(1+\theta)s}{s-2}}$ was given in \cite{LUO2};
see also \cite{BAL, CAB, MAY1} about decay rate estimates, for $F$ being convex and $\gamma(t)\sim\alpha / t^{\theta}$ (with $\theta,\alpha$ as above).

The problem of optimal decay rate for the limit case \eqref{gamma} turns out to be quite hard to study, especially if one attempts to
find an abundant of {\it slow solutions}, i.e. the solutions that decay (in the energy norm) {\it exactly} at the optimal rate (up to multiplicative constants).
The rate may be sensitive to changes of the values of $\alpha$ and the geometrical index $s$ of $F$, and so particularly refined analysis will be needed.

Recently, for \eqref {1.1} we obtained in \cite{LUO3} the energy decay rate
  \begin{equation}\label{nA1}
            E(t)=   \left\{
            \begin{array}{lll}
             O\left(t^{-\frac{2s}{s-2}}\right),  &{\rm if}~~ \alpha>\frac{2s}{s-2},\\[0.28cm]
           O\left(t^{-\frac{2s}{s-2}}\ln t\right), &{\rm if}~~\alpha=\frac{2s}{s-2},\\[0.28cm]
          O\left(t^{-\alpha}\right), &{\rm if}~~\frac{s+2}{s-2}\le\alpha<\frac{2s}{s-2}.
           \end{array}
        \right.
\end{equation}
 Moreover, we showed the optimality of the rate
  when $\alpha>\frac{2s}{s-2}$,
and left the problem open whether the energy has the upper decay rate $t^{-\frac{2s}{s-2}}$, also for
 $\alpha \in \left[\frac{s+2}{s-2},~\frac{2s}{s-2}\right]$ as in \eqref{nA}.

We first investigate the case of  $A$ being coercive (i.e. strongly positive). For concrete hyperbolic equations with coercive operators governing the linear parts, energy decay problems have been widely researched (cf., e.g.,  \cite{Abbas2015,Abbas20151,Akil,Batty2019,DAO,FAV,GOL,MAR,NAK,N2,XIA}). Yet very few studies have addressed the {\it optimality} of decay rates under time-dependent damping. All these papers mentioned here and above stimulate this work. Actually, we are able to prove that, for {\it any} $\alpha>0$, the energy decays at least as fast as $t^{-\alpha}$, and if the initial data $(u_0,u_1)\in \mathcal{H}_{1} \times \mathcal{H}\setminus\{(0,0)\}$, the energy decays at most as fast as $t^{-\alpha}$. This says that the energy of all solutions, except for one, decays at {\it exactly} the optimal rate $t^{-\alpha}$, indicating the absolute dominance of the slow solutions in amount.

Second, we study the general case when $A$  may not be coercive (so may have a nontrivial kernel like the negative Neumann Laplacian).  The optimality result about coercive $A$ in fact gives a negative answer to the aforementioned open problem posed in \cite{LUO3}(except for the case $\alpha=\frac{2s}{s-2}$). It tells us that for the general case, $O\left(t^{-\alpha}\right)$ is the best possible upper energy estimate for any $\alpha>0$, which is in contrast to \eqref{nA} (giving the optimal rate for the case $A=0$) when $\alpha<\frac{2s}{s-2}$. Therefore, the rate $t^{-\alpha}$ is optimal for $\alpha \in \left[\frac{s+2}{s-2},~\frac{2s}{s-2}\right)$ by \eqref{nA1}, and we see
 that $\frac{2s}{s-2}$ is another {\it critical value} of $\alpha$.

When the nonzero operator $A$ is not coercive and $\alpha \in\left[\frac{s+2}{s-2},~\frac{2s}{s-2}\right]$, is there still an abundant of {\it slow solutions} for which the energies decay exactly at the rate $t^{-\alpha}$? This issue is particularly challenging. Indeed, to obtain sharp lower energy estimation for $\alpha \le \frac{2s}{s-2}$ is harder than for $\alpha >\frac{2s}{s-2}$ considered in \cite{LUO3}. We finally succeed in demonstrating the dominance  of those slow solutions, for $\alpha\in\left[\frac{s+2}{s-2},~\frac{2s}{s-2}\right]$. Specifically, any initial data $(u_0,u_1)\in \mathcal{H}_{1} \times \mathcal{H}\setminus (\ker A\times\ker A)$ produce slow solutions. This is completely different from the case of $\gamma(\cdot)\equiv 1$ (considered in \cite{HAR,GHI}) or the case of $\alpha >\frac{2s}{s-2}$  (in \cite{LUO3}), where the initial data for slow solutions were found close to $\ker A$.

As a byproduct (Corollary \ref{cor}), for the linear equation \eqref{1.1-0} with $\gamma(t)= \alpha/t$, our result improves the existing related ones to a large extent.

It is worth noting that there seems to be little literature on exhibiting as many slow solutions as possible to various equations.
We here mention \cite{WIR1}(about the linear equation \eqref{1.1-0} with $A=-\Delta$ in $L^2(R^n)$ and $\gamma(t)$ decaying slower than $1/t$) where all nonzero solutions were shown to be slow solutions. Also in \cite{GHI} (for $\gamma\equiv 1$) the authors proved
the existence of a nonempty open set of initial data producing slow solutions.

In order to prove our results, we will  exploit flexibly the idea of asymptotic rate-sharpening and construct appropriate auxiliary functionals. The functionals are operator versions of those for ODEs originating from \cite{SU} (see also \cite{ATT5, AUJ, SEB} and the references therein). For showing the existence of slow solutions, we will establish favorable relationships between the initial data and coefficients; this is a key step. In the case of $A$ not being coercive, we will analyze the behaviors of the solutions by studying their range components in the space $(\ker A)^{\bot}$.
As will be seen, the optimal energy decay rate
is determined by the range component of the solution (rather than the kernel component), contrary to the case of $\alpha >\frac{2s}{s-2}$ considered in \cite{LUO3}.

By the way, in addition to their own importance, our results may also have reference value for studying the properties of the corresponding optimization algorithms. It is known that optimization algorithms (in particular, Nesterov's accelerated schemes) are associated with some second order damped ODEs (with time-dependent damping) (cf., e.g., \cite{SU}). There have been a lot of developments in the study of the asymptotic behaviors and decay rates for this class of continuous-time dynamical systems, aiming to provide insightful frameworks for the study of the properties of
optimization algorithms (cf., e.g., \cite{ATT1, ATT5,ATT6, AUJ, LAS1, SEB} and the references therein). From this perspective, the model \eqref{1.1} is related to the corresponding optimization algorithm for
the energy minimization problem:
$$\min\left\{\frac{1}{2}\|A^{1/2}u\|^2+F(u):~{u\in D(A^{1 / 2})}\right\}.$$
The prototype of the objective function in the problem above
is the Dirichlet-type energy $\frac{1}{2}\|\nabla u\|^2+F(u).$

The rest of this paper is organized as follows. In Section 2, we state basic assumptions on $F$, give a wellposedness result, and define two auxiliary functionals. Section 3 concerns \eqref{1.1} with coercive operator $A$, both upper and lower energy decay estimates being obtained. Section 4 is devoted to the case when $A$ is not coercive, and we present two theorems about lower energy decay estimates.  Moreover, in the last section, we give the proofs of some technical auxiliary lemmas which are used in Section 4.

\section{Basic assumptions and wellposedness}

Throughout this paper, $\mathcal{H}$ stands for a real Hilbert space, whose scalar product and norm are respectively denoted by $\langle\cdot\rangle$ and $||\cdot||$. And we define
 \begin{equation*}
   \mathcal{H}_1=D(A^{1/2})\quad \mbox{with norm}~~ \|v\|_{\mathcal{H}_1}:=\left(\|v\|^{2}+\|A^{1/2}v\|^{2}\right)^{1/2}.
 \end{equation*}

First we state the basic assumption on $F$.

\textbf{Assumption (H1)}

(1) $F:~\mathcal{H}_1\rightarrow R$ is nonnegative and differentiable with $F(0)=0$.

(2) The gradient $\nabla F:~ \mathcal{H}_1\rightarrow \mathcal{H}$ is locally Lipschitz continuous, i.e.
\begin{align*}
\|\nabla F(v)-\nabla F(w)\|\leq L_0(\|v\|_{\mathcal{H}_1},~\|w\|_{\mathcal{H}_1})\|v-w\|_{\mathcal{H}_1}, \quad \forall \  v,w \in \mathcal{H}_1,
\end{align*}
for some positive function $L_0$ on $R^+\times R^+$, which is bounded on bounded sets.

A function $u\in C([0,+\infty);\mathcal{H}_1)\cap C^{1}([0,+\infty);\mathcal{H})$ is called a {\it mild solution} of problem (\ref{1.1}), if it satisfies the integral equation
\begin{equation*}
  u(t)=S'(t)u_0+S(t)u_1-\int_0^tS(t-\tau)[\gamma(\tau) u'(\tau)+\nabla F(u(\tau))]d\tau,\quad t\ge 0.
\end{equation*}
Here $S(\cdot):~ [0,+\infty)\rightarrow \mathcal{L}(\mathcal{H})$ ~(the space of bounded linear operators on $\mathcal{H}$)~is a solution operator for the linear equation

\ \vspace{-10pt}
\begin{equation*}
  u''(t)+Au(t)=0, \quad t\ge 0,
\end{equation*}
with $S(0)=0$ and $S'(0)=I$ (the identity; the derivative being in the sense of strong topology).

In particular, $u$ is called a {\it strong solution} if $u\in C^1([0,+\infty);\mathcal{H}_1)\cap C^{2}([0,+\infty);\mathcal{H})$ and (\ref{1.1}) holds.

The following is a wellposedness result (cf. \cite{LUO} for a similar proof).

\begin{proposition}\label{pro}
Assume $A$ is a self-adjoint nonnegative operator on $\mathcal{H}$ with dense domain $D(A)$. Suppose that Assumption (H1) and \eqref{gamma} are satisfied.
Then, for every $(u_{0},u_{1})\in \mathcal{H}_1\times \mathcal{H}$, problem (\ref{1.1}) admits a unique global mild solution
\begin{align*}
u\in C([0,+\infty);\mathcal{H}_1)\cap C^{1}([0,+\infty);\mathcal{H}),
\end{align*}
which depends continuously on the initial data. In particular, $u$ is a strong solution if $(u_{0},u_{1})\in D(A)\times \mathcal{H}_1.$
Moreover, defining the energy
\begin{align*}
E(t):=\frac{1}{2}\|u'(t)\|^{2}+\frac{1}{2}\|A^{1/2}u(t)\|^{2}+F(u(t)),
\end{align*}
for strong solutions, one has
\begin{align*}
\frac{dE(t)}{dt}=-\gamma(t)\|u'(t)\|^2.
\end{align*}

\end{proposition}

The following assumption about $F$ is needed to study the decay rate of the energy.

\textbf{Assumption (H2)}

(1) There exists a constant $s>2$ such that, for $u\in \mathcal{H}_1,$
$$
sF(u)\leq \langle \nabla F(u), u\rangle,
$$

(2) There exists some positive constant $L_1$, such that $\forall v, w\in \mathcal{H}_1$,
$$
\|\nabla F(v)-\nabla F(w)\|\leq L_1\|v-w\|_{\mathcal{H}_1}\left(\|v\|^{s-2}_{\mathcal{H}_1}+\|w\|^{s-2}_{\mathcal{H}_1}\right).
$$

(H2)(1) is a flatness geometrical condition on $F$ in \cite{AUJ,CAB1,SEB,SU} for the case of $A=0$.

Let $\lambda>0$ be a constant and $\xi=\lambda(\lambda+1-\alpha)$. We define the following auxiliary energy:
$$
E_{\lambda}(t)=t^{2}F(u(t))+\frac{1}{2}\left\|\lambda u(t)+t u'(t)\right\|^{2}+\frac{\xi}{2}\left\|u(t)\right\|^{2}+\frac{t^2}{2}\|A^{1/2}u(t)\|^2.
$$
Write
\begin{align*}
a_u(t)&=tF(u(t)),\quad b_u(t)=\frac{1}{2 t}\left\|\lambda u(t)+t u'(t)\right\|^{2},\\
c_u(t)&=\frac{1}{2 t}\left\|u(t)\right\|^{2}, \quad d_u(t)=\frac{t}{2}\|A^{1/2}u(t)\|^2.
\end{align*}
Then, we have
$$E_{\lambda}(t)=t(a_u(t)+b_u(t)+\xi c_u(t)+d_u(t)).$$
Also, we define $H(t)=t^{\eta}E_{\lambda}(t)$, where $\eta$ is a constant.

\section{\eqref{1.1} with coercive operator $A$}

We consider equation \eqref{1.1} where $A$ is coercive, i.e.,
\begin{equation}\label{nu}
\exists \nu>0, \text { such that }\|A^{1 / 2} u\|^{2} \geq \nu\|u\|^{2}, \quad \forall u \in \mathcal{H}_1.
\end{equation}

First, we give an equality of $H'(t)$, whose proof is included in the proof process of \cite[Lemma 3.1]{LUO3}.
\begin{lemma}\label{lem1}
Let Assumption (H1) hold and $u:\left[t_{0},+\infty\right) \rightarrow H$ be a solution of \eqref{1.1}.
Then
\begin{equation}\label{H'}
\begin{aligned}
H'(t)=& t^{\eta}[(2+\eta)tF(u(t))-\lambda t\langle\nabla F(u(t)), u(t)\rangle+(2+\eta+ 2\lambda-2 \alpha) b_u(t)\\
&+\lambda(\lambda+1-\alpha)(\eta-2\lambda) c_u(t)+(2+\eta-2\lambda) d_u(t)].
\end{aligned}
\end{equation}
\end{lemma}

The next lemma illustrates the relationship between $E(t)$ and $E_{\lambda}(t)$.

\begin{lemma}\label{lem2}
Assume (H1), \eqref{nu} and
\begin{equation}\label{t01}
t^2_1\ge\max \left\{\frac{2\lambda}{\nu}|\alpha-1|,\frac{\lambda}{\nu}|3\lambda+1-\alpha|\right\}.
\end{equation}
If $u$ is a solution to \eqref{1.1}, then
\begin{equation}\label{Elambdalower}
\frac{t^2}{2}E(t)\leq E_{\lambda}(t)\leq 2t^2E(t), \quad\forall t\ge t_1.
\end{equation}
\end{lemma}

\begin{proof}
On one hand,
\begin{align*}
\frac{1}{2}\left\|\lambda u(t)+t u'(t)\right\|^{2}&=\frac{1}{2}\left(\lambda^2\|u(t)\|^2+2\lambda t\langle u(t),u'(t)\rangle+t^2\|u'(t)\|^2\right)\\
&\geq\frac{t^2}{4}\|u'(t)\|^2-\frac{\lambda^2}{2}\|u(t)\|^2,
\end{align*}
which implies
\begin{align*}
E_{\lambda}(t)&=t^{2}F(u(t))+\frac{1}{2}\left\|\lambda u(t)+t u'(t)\right\|^{2}+\frac{\xi}{2}\left\|u(t)\right\|^{2}+\frac{t^2}{2}\|A^{1/2}u(t)\|^2\\
&\ge t^{2}F(u(t))+\frac{t^2}{4}\|u'(t)\|^2+\frac{t^2}{2}\|A^{1/2}u(t)\|^2-\frac{|\lambda^2-\xi|}{2\nu}\|A^{1/2}u(t)\|^2\\
&\ge\frac{t^{2}}{2}F(u(t))+\frac{t^2}{4}\|u'(t)\|^2+\frac{t^2}{4}\|A^{1/2}u(t)\|^2,
\end{align*}
where the last inequality is from $\xi=\lambda(\lambda+1-\alpha)$ and \eqref{t01}. That is
$$
\frac{t^2}{2}E(t)\leq E_{\lambda}(t).
$$

On the other hand, by $\xi=\lambda(\lambda+1-\alpha)$ and \eqref{t01}, we deduce
\begin{align*}
E_{\lambda}(t)&=t^{2}F(u(t))+\frac{1}{2}\left\|\lambda u(t)+t u'(t)\right\|^{2}+\frac{\xi}{2}\left\|u(t)\right\|^{2}+\frac{t^2}{2}\|A^{1/2}u(t)\|^2\\
&\leq t^{2}F(u(t))+t^2\|u'(t)\|^2+\frac{|2\lambda^2+\xi|}{2\nu}\left\|A^{1/2}u(t)\right\|^{2}+\frac{t^2}{2}\|A^{1/2}u(t)\|^2\\
&\leq 2t^{2}F(u(t))+t^2\|u'(t)\|^2+t^2\|A^{1/2}u(t)\|^2.
\end{align*}

Hence, we complete the proof.
\end{proof}

The third lemma is a preliminary result about evaluating the energy decay along each trajectory $u(\cdot)$.

\begin{lemma}\label{lem3}
Let Assumptions (H1) and (H2)(1) hold. Let $\eta=\alpha-2-2\varepsilon$ and $\lambda=\frac{\alpha}{2}+\varepsilon$ with some small positive constant $\varepsilon$. In addition, we assume that \eqref{gamma} and \eqref{nu} hold with
\begin{equation}\label{t02}
t^{2}_{2,\varepsilon}\ge\frac{\lambda(1+2\varepsilon)}{2\nu \varepsilon}|\alpha-2-2\varepsilon|+t_1,
\end{equation}
where $t_1$ is defined in \eqref{t01}.
Then, $$
E(t)\leq 2H(t_{2,\varepsilon})t^{-\alpha+2\varepsilon},\quad\forall t\ge t_{2,\varepsilon}.
$$
\end{lemma}

\begin{proof}

Taking Assumption (H2)(1) into consideration, we infer from \eqref{H'} that
\begin{equation}\label{H'1}
\begin{aligned}
H'(t)\leq &t^{\eta}((2+\eta-s \lambda) a_u(t)+(2+\eta+ 2\lambda-2 \alpha) b_u(t)\\
&+\lambda(\eta-2\lambda)(\lambda+1-\alpha) c_u(t)+(2+\eta-2\lambda) d_u(t)).
\end{aligned}
\end{equation}
Next, since $\eta=\alpha-2-2\varepsilon$ and $\lambda=\frac{\alpha}{2}+\varepsilon$, we deduce
\begin{equation}\label{2+eta-s}
\begin{aligned}
2+\eta-s \lambda&=\alpha-\frac{s}{2}\alpha-(s+2)\varepsilon<0,\\
2+\eta+ 2\lambda-2 \alpha&=\alpha-2\varepsilon+\alpha+2\varepsilon-2\alpha=0,\\
\lambda(\lambda+1-\alpha)(\eta-2\lambda)&=\lambda(\alpha-2-2\varepsilon)(1+2\varepsilon),\\
2+\eta-2\lambda&=\alpha-2\varepsilon-\alpha-2\varepsilon=-4\varepsilon,
\end{aligned}
\end{equation}
which yields, by \eqref{H'1},
\begin{equation}\label{H'a>2}
H'(t)\leq t^{\eta}\left[\frac{2-s}{2}\alpha a_u(t)+\lambda(\alpha-2-2\varepsilon)(1+2\varepsilon)c_u(t)-4\varepsilon d_u(t)\right].
\end{equation}
Using \eqref{nu} and \eqref{t02} gives that
\begin{align*}
\lambda(\alpha-2-2\varepsilon)(1+2\varepsilon)c_u(t)-4\varepsilon d_u(t)&\leq\frac{\lambda}{2t}|\alpha-2-2\varepsilon|(1+2\varepsilon)\|u(t)\|^2-2\varepsilon t\|A^{1/2}u(t)\|^2\\
&\leq-\varepsilon t\|A^{1/2}u(t)\|^2,
\end{align*}
from which we obtain $H'(t)\leq0$ for any $t\ge t_{2,\varepsilon}$. That is,
$
E_{\lambda}(t)=t^{-\eta}H(t)\le t^{-\eta}H(t_{2,\varepsilon}).
$
Hence, it follows from \eqref{Elambdalower} that
$$
E(t)\leq 2E_{\lambda}(t)t^{-2}\le2 H(t_{3,\varepsilon})t^{-\eta-2}=2H(t_{3,\varepsilon})t^{-\alpha+2\varepsilon},\quad \forall t\ge t_{2,\varepsilon}.
$$
Therefore, we complete the proof.
\end{proof}

Now, we are ready to present the two main results in the section. The first one concerns upper decay estimate.

\begin{theorem}\label{thm1}
Let Assumptions (H1) and (H2)(1) hold. Assume that \eqref{gamma} and \eqref{nu} hold.
Let $(u_{0},u_{1})\in \mathcal{H}_1\times \mathcal{H} $, and $u(t)$ be the unique global mild solution of problem \eqref{1.1}.
Then, if $t_{3,\varepsilon}$ is large enough such that it satisfies \eqref{t01}, \eqref{t02}, and
\begin{equation}\label{t03}
t_{3,\varepsilon}\ge\frac{2\alpha|\alpha-2|}{\nu}+\sqrt{\frac{\alpha+2}{3\nu}},
\end{equation}
the solution energy satisfies
\begin{equation}\label{4Ht0}
E(t)\leq 8t^{\alpha}_{3,\varepsilon}E(t_0)t^{-\alpha},\quad\forall t\ge t_{3,\varepsilon}.
\end{equation}
\end{theorem}

\begin{proof}
Throughout the proof, $\eta=\alpha-2$ and $\lambda=\frac{\alpha}{2}$. From \eqref{2+eta-s} with $\varepsilon=0$, we obtain
\begin{equation}\label{H'alpha}
H'(t)\leq \frac{\alpha(\alpha-2)}{2}t^{\eta} c_u(t)=\frac{\alpha(\alpha-2)}{4}t^{\eta-1}\|u(t)\|^2.
\end{equation}

\noindent {\bf Case I}:  \ {\it $\alpha\leq2$}. In this case, \eqref{H'alpha} indicates $H'(t)\leq0$ , and so $H(t)\leq H(t_{3,\varepsilon})$. Thus,
$$
E_{\lambda}(t)=H(t)t^{-\eta}\le H(t_{3,\varepsilon}) t^{-\eta}.
$$
Combining this with \eqref{Elambdalower}, we obtain
$$
E(t)\leq2E_{\lambda}(t)t^{-2}\leq 2H(t_{3,\varepsilon})t^{-\eta-2}=2H(t_{3,\varepsilon})t^{-\alpha}.
$$

\vspace{0.4cm}

\noindent {\bf Case II}: \  {\it $\alpha>2$}. We denote by $\tilde{E}_\lambda(t)$ (resp. $\tilde{H}(t)$) the functional $E_\lambda(t)$ (resp. $H(t)$) with the parameters $\eta=\alpha-2-2\varepsilon$ and $\lambda=\frac{\alpha}{2}+\varepsilon$ (as in Lemma \ref{lem3}).

From the Cauchy-Schwarz inequality, \eqref{nu} and \eqref{t03}, we have, for $\varepsilon<\frac{1}{3}$,
\begin{align*}
&\varepsilon t\langle u(t), u'(t)\rangle+\left(\varepsilon^2+\frac{\alpha+1}{2}\varepsilon\right)\|u(t)\|^2\\
\leq& \frac{t^2}{2}\|u'(t)\|^2+\frac{\varepsilon^2}{2}\|u(t)\|^2+\left(\varepsilon^2+\frac{\alpha+1}{2}\varepsilon\right)\|u(t)\|^2\\
\leq& \frac{t^2}{2}\|u'(t)\|^2+\frac{\alpha+2}{6}\|u(t)\|^2\\
\leq& \frac{t^2}{2}\|u'(t)\|^2+\frac{t^2}{2}\|A^{1/2}u(t)\|^2,
\end{align*}
which gives
\begin{equation}\label{tildeEH}
\tilde{E}_\lambda(t)\leq 2E_{\lambda}(t) \,\mbox{~and}\,\, \tilde{H}(t)\leq 2H(t),\, \mbox{for all}~ t\ge t_{3,\varepsilon}.
\end{equation}
Therefore, using the above inequality, Lemmas \ref{lem3}, \eqref{nu} and \eqref{H'alpha}, we have
\begin{align*}
H'(t)&\leq\frac{\alpha(\alpha-2)}{4}t^{\eta-1}\|u(t)\|^2
\leq\frac{\alpha(\alpha-2)}{\nu}\tilde{H}(t_{3,\varepsilon})t^{-3+2\varepsilon}
\leq\frac{2\alpha(\alpha-2)}{\nu}H(t_{3,\varepsilon})t^{-2},
\end{align*}
by $\varepsilon<1/3$. Accordingly,
\begin{align*}
H(t)\leq H(t_{3,\varepsilon})-\frac{2\alpha(\alpha-2)}{\nu}H(t_{3,\varepsilon})t^{-1}
+\frac{2\alpha(\alpha-2)}{\nu}H(t_{3,\varepsilon})t^{-1}_{3,\varepsilon}\leq 2H(t_{3,\varepsilon}),
\end{align*}
by \eqref{t03}. As a result, we obtain
$
E_{\lambda}(t)=H(t)t^{-\eta}\le 2 H(t_{3,\varepsilon})t^{-\eta}.
$
From $\eta=\alpha-2$ and \eqref{Elambdalower}, we derive
\begin{equation}\label{u1}
  E(t)\leq4 H(t_{3,\varepsilon})t^{-\eta-2}=4 H(t_{3,\varepsilon})t^{-\alpha}, \quad \forall t\ge t_{3,\varepsilon}.
\end{equation}
Then, it follows from Lemma \ref{lem2} that
$$
H(t)=t^{\eta}E_{\lambda}(t)\leq 2t^{\eta+2}E(t)=2t^{\alpha}E(t), \quad \forall t\ge t_{3,\varepsilon}.
$$
Hence
$$
H(t_{3,\varepsilon})\leq 2t^{\alpha}_{3,\varepsilon}E(t_{3,\varepsilon})\leq 2t^{\alpha}_{3,\varepsilon}E(t_0),
$$
where the last inequality is from the non-increasing property of $E(t)$. Combining this and \eqref{u1} shows \eqref{4Ht0}.
\end{proof}

The second theorem shows the existence of slow solutions, which means that the decay rate in Theorem \ref{thm1} is optimal.

\begin{theorem}\label{thm2}
Let Assumptions (H1) and (H2) hold. Assume that \eqref{gamma} and \eqref{nu} are satisfied. Then, for every $(u_0,u_1)\in \mathcal{H}_{1} \times \mathcal{H}\setminus\{(0,0)\}$, there exists $t_4\ge t_{3,\varepsilon}$,
where $t_{3,\varepsilon}$ is defined in Theorem \ref{thm1}, such that the unique global solution of problem \eqref{1.1} satisfies
\begin{equation}\label{t-alpha}
\|u'(t)\|^2+\|A^{1/2}u(t)\|^2+F(u(t))\ge \frac{t^{-\alpha}_{4}}{4}t^{2\alpha}_{0}E(t_0) t^{-\alpha},\quad \forall t\ge t_4.
\end{equation}
\end{theorem}

\begin{proof}

First, we consider the case $F=0$, and let $\eta=\alpha-2$ and $\lambda=\frac{\alpha}{2}$. Then, from \eqref{2+eta-s} with $\varepsilon=0$ and \eqref{H'}, we obtain
$$
H'(t)= \frac{\alpha(\alpha-2)}{2}t^{\eta} c_u(t)=\frac{\alpha(\alpha-2)}{4}t^{\eta-1}\|u(t)\|^2.
$$
Obviously, the above equality equals zero if $\alpha=2$.

When $\alpha >2$, we have $H'(t)\geq0$, and so $H(t)\geq H(t_{3,\varepsilon})$. That is,
$
E_{\lambda}(t)\geq H(t_{3,\varepsilon})t^{-\eta}.
$
Hence, using \eqref{Elambdalower} gives
\begin{align*}
\|u'(t)\|^2+\|A^{1/2}u(t)\|^2\geq E_{\lambda}(t) t^{-2}\geq H(t_{3,\varepsilon})t^{-\alpha}.
\end{align*}

When $\alpha<2$, from $\eta=\alpha-2$, \eqref{nu} and \eqref{H'alpha}, we find
\begin{align*}
H(t)&=H(t_{3,\varepsilon})+\frac{\alpha(\alpha-2)}{4}\int^{t}_{t_{3,\varepsilon}}\tau^{\eta-1}\|u(\tau)\|^2d\tau\\
&\geq H(t_{3,\varepsilon})+\frac{\alpha(\alpha-2)}{4\nu}\int^{t}_{t_{3,\varepsilon}}\tau^{\eta-1}\|A^{1/2}u(\tau)\|^2d\tau\\
&\geq H(t_{3,\varepsilon})+\frac{\alpha(\alpha-2)}{\nu}H(t_{3,\varepsilon})\int^{t}_{t_{3,\varepsilon}}\tau^{\eta-1-\alpha}d\tau\\
&\geq H(t_{3,\varepsilon})+\frac{\alpha(\alpha-2)}{-2\nu}H(t_{3,\varepsilon})t^{-2}+\frac{\alpha(\alpha-2)}{2\nu}
H(t_{3,\varepsilon})t^{-2}_{3,\varepsilon}\\
&\geq H(t_{3,\varepsilon})+\frac{\alpha(\alpha-2)}{2\nu}H(t_{3,\varepsilon})t^{-2}_{3,\varepsilon}\\
&\geq \frac{1}{2}H(t_{3,\varepsilon}),
\end{align*}
where the last inequality is from \eqref{t03}. This and $H(t)=t^{\eta}E_{\lambda}(t)$ together yield
$$
E_{\lambda}(t)=H(t)t^{-\eta}\geq\frac{1}{2}H(t_{3,\varepsilon})t^{-\eta}.
$$
Then, from Lemma \ref{lem2}, we obtain
\begin{equation}\label{Eu}
\|u'(t)\|^2+\|A^{1/2}u(t)\|^2\geq t^{-2}E_{\lambda}(t)\geq\frac{H(t_{3,\varepsilon})}{2}t^{-\alpha}.
\end{equation}
Hence, it follows from \eqref{1.1} that
$$
E(t)=-\gamma(t)\|u'(t)\|^2\ge -2\gamma(t)E(t),
$$
from which we obtain
$$
E(t)\ge t^{2\alpha}_0E(t_0)t^{-2\alpha}, \quad \forall t\ge t_0.
$$
Combining this and Lemma \ref{lem2}, we get
$$
H(t_{3,\varepsilon})=t^{\eta}_{3,\varepsilon}E_{\lambda}(t_{3,\varepsilon})\ge \frac{t^{\eta+2}_{3,\varepsilon}}{2}E(t_{3,\varepsilon})=\frac{t^{\alpha}_{3,\varepsilon}}{2}E(t_{3,\varepsilon})\ge \frac{t^{-\alpha}_{3,\varepsilon}}{2}t^{2\alpha}_{0}E(t_0).
$$
This and \eqref{Eu} give
\begin{equation}\label{Eut3}
\|u'(t)\|^2+\|A^{1/2}u(t)\|^2\geq \frac{t^{-\alpha}_{3,\varepsilon}}{4}t^{2\alpha}_{0}E(t_0)t^{-\alpha}, \quad \forall t\ge t_{3,\varepsilon},
\end{equation}
and hence prove \eqref{t-alpha}.

Next, consider $F\neq0$, and let $\eta=\alpha-2$ and $\lambda=\frac{\alpha}{2}$.
Hence, it follows from \eqref{H'} and \eqref{2+eta-s} with $\varepsilon=0$ that
\begin{align*}
H'(t)&=t^{\eta}\left(\alpha tF(u(t))-\lambda t\langle\nabla F(u(t)), u(t)\rangle+\frac{\alpha(\alpha-2)}{2}c_{u}(t)\right)\\
&\ge -\lambda||\nabla F(u(t))||\cdot|| u(t)||t^{\eta+1}+\frac{\alpha(\alpha-2)}{4}||u(t)||^2t^{\eta-1}.
\end{align*}
Using (H2)(2) with $w=0$ and \eqref{nu}, we obtain, $\forall v\in \mathcal{H}_1$,
$$
\|\nabla F(v)\|\leq L_1\|v\|^{s-1}_{\mathcal{H}_1}\leq 2^{\frac{s-1}{2}}L_1\left(\|v(t)\|^{s-1}+\|A^{1/2}v(t)\|^{s-1}\right)\leq C\|A^{1/2}v(t)\|^{s-1}.
$$
Accordingly, we deduce
\begin{equation}\label{H'C0}
H'(t)\ge -C_0\|A^{1/2}u(t)\|^{s}t^{\eta+1}+\frac{\alpha(\alpha-2)}{4}||u(t)||^2t^{\eta-1},
\end{equation}
with a positive constant $C_0$.

Later, we estimate $H(t)$. Combining \eqref{Elambdalower} and \eqref{u1}, we deduce
\begin{equation}\label{Ht1t2}
H(t)=t^{\eta}E_{\lambda}(t)\le 2t^{\eta+2}E(t)\leq 8H(t_{3,\varepsilon}), \quad \forall t\ge t_{3,\varepsilon}.
\end{equation}
Then, we distinguish two cases.

\vspace{0.4cm}

\noindent {\bf Case I}: {\it $\alpha\ge2$}. In this case, for $t\ge t_4\ge t_{3,\varepsilon}$, it follows from \eqref{u1} and \eqref{H'C0} that
\begin{align*}
H'(t)\geq -C_0\|A^{1/2}u(t)\|^{s}t^{\eta+1}\ge -C\left[H(t_4)t^{-\alpha}\right]^{\frac{s}{2}}t^{\alpha-1}.
\end{align*}
Then, integrate the inequality, and make use of \eqref{tildeEH} and \eqref{Ht1t2}. We get
\begin{align*}
H(t)\geq H(t_4) -C\left[H(t_4)\right]^{\frac{s}{2}}\int^{t}_{t_4}\tau^{\alpha-\frac{s}{2}\alpha-1}d\tau
\geq H(t_4)-CH(t_4)\left[8H(t_{3,\varepsilon})\right]^{\frac{s-2}{2}}t^{-\frac{s-2}{2}\alpha}_4.
\end{align*}
Then, for every $(u_0,u_1)\in \mathcal{H}_{1} \times \mathcal{H}\setminus\{(0,0)\}$, we can choose $t_4> t_{3,\varepsilon}$ large enough such that
$$
C\left[8H(t_{3,\varepsilon})\right]^{\frac{s-2}{2}}t^{-\frac{s-2}{2}\alpha}_4<\frac{1}{2}.
$$
Hence, $H(t)\geq \frac{1}{2}H(t_4)$. Proceeding similarly as in the case when $F=0$ and $\alpha>2$, we can get the conclusion.

\vspace{0.4cm}

\noindent {\bf Case II}: {\it $\alpha<2$}.
 Using \eqref{nu} and \eqref{H'C0}, we deduce
\begin{align*}
H'(t)\geq&  -C_0\|A^{1/2}u(t)\|^{s}t^{\eta+1}+\frac{\alpha(\alpha-2)}{4\nu}||A^{1/2}u(t)||^2t^{\eta-1}.
\end{align*}
Hence, integrating the inequality we obtain
\begin{align*}
H(t)\geq&H(t_4)  -C_0\int^{t}_{t_4}\|A^{1/2}u(\tau)\|^{s}\tau^{\eta+1}d\tau+
\int^{t}_{t_4}\frac{\alpha(\alpha-2)}{4\nu}||A^{1/2}u(\tau)||^2\tau^{\eta-1}d\tau\\
\geq&H(t_4) -C\int^{t}_{t_4}\left[ H(t_4)\tau^{-\alpha}\right]^{\frac{s}{2}}\tau^{\alpha-1}d\tau
-\int^{t}_{t_4}\frac{4\alpha}{\nu} H(t_4)\tau^{-3}d\tau\\
\ge &H(t_4)-CH(t_4)\left[8H(t_{3,\varepsilon})\right]^{\frac{s-2}{s}}t^{-\frac{s-2}{2}\alpha}_4-\frac{2\alpha}{\nu} H(t_4)t^{-2}_{4},
\end{align*}
by Lemma \ref{lem3} and \eqref{tildeEH}. Then, for every $(u_0,u_1)\in \mathcal{H}_{1} \times \mathcal{H}\setminus\{(0,0)\}$, we can choose $t_4> t_{3,\varepsilon}$ large enough such that
$$
C\left[8H(t_{3,\varepsilon})\right]^{\frac{s-2}{s}}t^{-\frac{s-2}{2}\alpha}_4\leq \frac{1}{4},~\mbox{and}\,\,\,\frac{2\alpha}{\nu} t^{-2}_{4}\leq \frac{1}{4}.
$$
Consequently, $H(t)\ge \frac{1}{2}H(t_4)$, and we complete the proof.
\end{proof}

\section{The case of $A$ being not coercive}

In this section, $A$ is not coercive and we assume
\begin{equation}\label{nu3}
\exists \nu>0, \text { such that }\|A^{1 / 2} u\|^{2} \geq \nu\|u\|^{2}, \quad \forall u \in D(A^{1 / 2}) \cap \operatorname{ker}(A)^{\perp}.
\end{equation}
The above assumption is also called ellipticity-like condition; see \cite{CAB}.

We need more assumptions on $F$ to study the decay rate of energy.

\textbf{Assumption (H3)}

(1) There exists a function $M_0: R \rightarrow R$ which is bounded on bounded sets, such that
$$
\|u\|^{s}\leq M_0(\|A^{1/2}u\|)(\|A^{1/2}u\|^{2}+F(u)),\quad\forall u\in \mathcal{H}_1,
$$
where $s$ is from Assumption (H2).

(2) There exists a real number $R>0$ such that
$$
\|Q\nabla F(u)\| \leq R(\|u\|^{s-2}+\|A^{1 / 2} u\|^{s-2})\|A^{1 / 2} u\|,\quad\forall u\in \mathcal{H}_1,
$$
where $Q=I-P$, and $P: \mathcal{H} \rightarrow  \ker{A}$ denotes the orthogonal projection on $\ker{A}$.
\vspace{0.5cm}

From Proposition $\ref{pro}$, we know that
$E(t)$
is a non-increasing function. Hence
$$\frac{1}{2}\|A^{1/2}u(t)\|^{2}\leq E(t)\leq E(t_0),$$
which means that the function $M_0(\|A^{1/2}u\|)$ in (H3)(1) can be treated as a constant. (H3)(2) has been used in \cite{LUO3} to obtain lower energy decay estimates for \eqref{1.1} with $\alpha>\frac{2s}{s-2}$.

First, we state the result about upper decay estimates from \cite[Theorem 2.3]{LUO3}.

\begin{theorem}\label{thm2.3}
Suppose that Assumptions (H1), (H2)(1), (H3)(1) and \eqref{gamma} hold with $\frac{s+2}{s-2}\le\alpha<\frac{2s}{s-2}$. Let $(u_0, u_1)\in \mathcal{H}_1\times \mathcal{H}$, and $u(t)$ be the solution of \eqref{1.1}. Then
$$
E(t)\leq M(E(t_0))t^{-\alpha},\quad \forall t\ge t_0,
$$
where M is some positive function being bounded on bounded sets.
\end{theorem}

Then, we will establish two theorems regarding lower decay estimates of the energy. To the end, we need to do some preparatory work.

Let $P,Q$ be as in (H3)(2); they are the orthogonal projections on the kernel and the range of $A$, respectively. Set $w=Pu$ and $v=Qu$. It is clear that
$u=w+v$, and $w$ and $v$ solve, respectively,
\begin{eqnarray}\label{wt}
\left\{\begin{array}{ll}
w''(t)+\gamma(t)w'(t)+\nabla F(w(t))+g(t)=0, \quad \forall t\geq t_0,\\[0.36cm]
w(t_0)=P u_{0}, ~w'(t_0)=P u_{1},
\end{array} \right.
\end{eqnarray}
with $g(t)=P\nabla F(u(t))-\nabla F(w(t))$, and
\begin{eqnarray}\label{vt}
\left\{\begin{array}{ll}
v''(t)+A v(t)+\gamma(t)v'(t)+Q\nabla F(u(t))=0, \quad \forall t\geq t_0,\\[0.36cm]
v(t_0)=Q u_{0}, ~v'(t_0)=Q u_{1}.
\end{array} \right.
\end{eqnarray}
Moreover, we have
$
\|A^{1/2}u(t)\|^2=\|A^{1/2}v(t)\|^2,
$
and
\begin{equation}\label{vnu}
\|A^{1 / 2} v(t)\|^{2} \geq \nu\|v(t)\|^{2}.
\end{equation}
For \eqref{vt}, we define
\begin{equation}\label{venergy}
E_v(t)=\|v'(t)\|^2+\|A^{1/2}v(t)\|^2.
\end{equation}
Also, like the energy functions $E_{\lambda}(t)$ and $H(t)$ in Section 3, we define
$$
E_{\lambda,v}(t)=\frac{1}{2}\left\|\lambda v(t)+t v'(t)\right\|^{2}+\frac{\xi}{2}\left\|v(t)\right\|^{2}+\frac{t^2}{2}\|A^{1/2}v(t)\|^2,~\mbox{and}
$$
\begin{equation}\label{Hv}
H_v(t)=t^{\eta}E_{\lambda,v}(t)+\int^{t}_{t_0}\left\langle\lambda v(\tau)+\tau v'(\tau), \tau^{\eta+1} Q\nabla F(u(\tau))\right\rangle d \tau,
\end{equation}
where $\lambda>0$ and $\eta$ are constant, and $\xi=\lambda(\lambda+1-\alpha)$. Put
$$
b_v(t)=\frac{1}{2 t}\left\|\lambda v(t)+t v'(t)\right\|^{2},\,\, c_v(t)=\frac{1}{2 t}\left\|v(t)\right\|^{2},\,\,
d_v(t)=\frac{t}{2}\|A^{1/2}v(t)\|^2,
$$
we have
$E_{\lambda,v}(t)=t(b_v(t)+\xi c_v(t)+d_v(t)).$

Next, we give several lemmas concerning $v$.
\begin{lemma}\label{lem0}
Assume (H1). Let $v:\left[t_{0},+\infty\right) \rightarrow \mathcal{H}$ be a solution of \eqref{vt}. Then
\begin{equation}\label{Hv'}
\begin{aligned}
H'_v(t)= t^{\eta}((2+\eta+ 2\lambda-2 \alpha) b_v(t)+\lambda(\eta-2\lambda)(\lambda+1-\alpha) c_v(t)+(2+\eta-2 \lambda) d_v(t)).
\end{aligned}
\end{equation}
\end{lemma}

\begin{proof}
Let us differentiate the energy $E_{\lambda,v}(t)$, getting
\begin{align*}
E'_{\lambda,v}(t)=&\left\langle\lambda v(t)+tv'(t), (\lambda+1) v'(t)+tv''(t)\right\rangle+\xi\left\langle v(t), v'(t)\right\rangle\\
&+t\|A^{1/2}v(t)\|^2+t^2\langle Av(t), v'(t)\rangle.
\end{align*}
Since $v(t)$ is a solution of the equation \eqref{vt}, we obtain
$$(\lambda+1) v'(t)+t v''(t)=(\lambda+1-\alpha) v'(t)-tAv(t)-tQ\nabla F(u(t)),$$
which implies
\begin{align*}
E'_{\lambda,v}(t)=& (\xi+\lambda(\lambda+1-\alpha))\left\langle v(t),v'(t)\right\rangle+(1-\lambda)t\|A^{1/2}v(t)\|^2 \\
&+t(\lambda+1-\alpha)\|v'(t)\|^2-\langle\lambda v(t)+tv'(t),tQ\nabla F(u(t))\rangle.
\end{align*}
Noting
\begin{align*}
\frac{1}{t}\left\|\lambda v(t)+t v'(t)\right\|^{2}= t\|v'(t)\|^{2}+2 \lambda\left\langle v'(t), v(t)\right\rangle +\frac{\lambda^{2}}{t}\left\|v(t)\right\|^{2},
\end{align*}
and $\xi=\lambda(\lambda+1-\alpha)$, we deduce
\begin{align*}
E'_{\lambda,v}(t)
=&(2\lambda+2-2\alpha)b_v(t)-2\lambda^2(\lambda+1-\alpha)c_v(t)\\
&+(2-2\lambda)d_u(t)-\langle\lambda v(t)+tv'(t),tQ\nabla F(u(t))\rangle.
\end{align*}
Since
$$
H'_v(t)=\eta t^{\eta-1}E_{\lambda,v}(t)+t^{\eta}E^{'}_{\lambda,v}(t)+\langle\lambda v(t) +tv'(t),t^{\eta+1}Q\nabla F(u(t))\rangle,
$$
we obtain \eqref{Hv'}.
\end{proof}

The next lemma is similar to Lemma \ref{lem2}.
\begin{lemma}\label{lem4}
Assume (H1), \eqref{gamma} and \eqref{nu3} hold.
If $v$ is a solution to \eqref{vt}, then
$$
\frac{t^2}{2}E_{v}(t)\leq E_{\lambda,v}(t)\leq 2t^2E_{v}(t), \quad\forall t\ge t_1,
$$
where $t_1$ is defined in \eqref{t01} and $E_{v}(t)$ is as in \eqref{venergy}.
\end{lemma}

\begin{proof}
Argue in the same way as for Lemma \ref{lem2}.
\end{proof}

The last lemma is to estimate $E_{\lambda, v}(t)$ from below.

\begin{lemma}\label{lem6}
Assume that (H1), (H2)(1), (H3)(1), \eqref{gamma}, \eqref{t01} and \eqref{nu3} hold. Let $\eta=\alpha-2$ and $\lambda=\frac{\alpha}{2}$.
Let $(u_{0},u_{1})\in \mathcal{H}_1\times \mathcal{H} $, and $u(t)$ be the unique global mild solution of problem \eqref{1.1}.
Then, $\forall t\ge T\ge t_1$, we have
\begin{equation}\label{tvare1}
\begin{aligned}
&t^{\eta}E_{\lambda,v}(t)-T^{\eta}E_{\lambda,v}(T)\\
=&\int^{t}_{T}\frac{\lambda(\alpha-2)}{2}||v(\tau)||^2\tau^{\eta-1}d \tau-\int^{t}_{T}\left\langle\lambda v(\tau)+\tau v'(\tau), \tau^{\eta+1} Q\nabla F(u(\tau))\right\rangle d \tau.
\end{aligned}
\end{equation}
\end{lemma}

\begin{proof}
Since $\eta=\alpha-2$ and $\lambda=\frac{\alpha}{2}$, it follows from \eqref{Hv'} that
\begin{align*}
2+\eta+ 2\lambda&-2 \alpha=\alpha+\alpha-2\alpha=0,~2+\eta-2\lambda=0,\\
&\lambda(\lambda+1-\alpha)(\eta-2\lambda)=\lambda(\alpha-2),
\end{align*}
which yields
$$H'_v(t)=\frac{\lambda(\alpha-2)}{2}||v(t)||^2t^{\eta-1},\quad \forall t\ge t_1.$$
This and \eqref{Hv} give \eqref{tvare1}, and we complete the proof.
\end{proof}

Subsequently, we distinguish the case of $\alpha\in \left[\frac{s+2}{s-2}, 2\right]$ from that of $\alpha\in (2, 2s/(s-2)]$. The results for the case of $\alpha> 2s/(s-2)$ have been obtained
in \cite{LUO3}.

\subsection{$\frac{s+2}{s-2}\le\alpha\le2$}

In this case, $s\ge6$ from
$
\displaystyle\frac{s+2}{s-2}\le\alpha\le2.
$
We will use $\bar{C}_0, \bar{C}_1, \bar{C}_2, \cdot\cdot\cdot$ to denote generic positive constants independent of $t$. First, we give a lemma about the energy decay along each trajectory $u(\cdot)$, which is similar to Lemma \ref{lem3}. The proofs of this lemma and the following one will be given in Section 5.

\begin{lemma}\label{th2'}
Suppose that Assumptions (H1), (H2)(1), (H3)(1) and \eqref{gamma} hold. Let $(u_0, u_1)\in \mathcal{H}_1\times \mathcal{H}$, and
$u(t)$ be the solution of \eqref{1.1}.

\noindent 1.~If $\alpha\leq2$, then
\begin{equation}\label{ualpha<2}
F(u(t))+\|A^{1/2}v(t)\|^2\leq 2H^*(t_0)t^{-\alpha}, \quad \forall t\ge t_0.
\end{equation}
Here $H^*(t_0)$ stands for $H(t_0)$ with $\eta=\alpha-2$ and $\lambda=\frac{\alpha}{2}$.

\noindent (2).~If $2<\alpha\le\frac{2s}{s-2}$, then
\begin{equation}\label{ualpha>2}
F(u(t))+\|A^{1/2}v(t)\|^2\leq 4\tilde H^*(t_0)t^{-\alpha+2\varepsilon}+C_\varepsilon
 t^{\alpha-\frac{2s}{s-2}-2\varepsilon}_0t^{-\alpha+2\varepsilon}, \quad \forall t\ge t_0.
\end{equation}
Here $C_\varepsilon$ is a positive constant independent of $E(t_0)$ and $t_0$, $\tilde H^*(t_0)$ stands for $\tilde H(t_0)$ with $\eta=\alpha-2-2\varepsilon$ and $\lambda=\frac{\alpha}{2}+\varepsilon$, and
\begin{equation}\label{vare}
0<\varepsilon<\min\left\{\frac{1}{4(s-2)},\,~\frac{1}{2}\right\} ~~\mbox{fixed}.
\end{equation}
\end{lemma}

The estimates \eqref{ualpha<2} and \eqref{ualpha>2} establish the relationship between the decay rates of $\|A^{1/2}v(t)\|$ and $H^*(t_0)$ or $\tilde H^*(t_0)$. With these results, we can illustrate the relationship between the decay rates of $E_v$ and $E_v(t_1)$ for small initial data, showing how the coefficients depend on the initial data, which is essentially important in the following proof.
\begin{lemma}\label{lem8}
Let (H1), (H2), (H3), \eqref{gamma} and \eqref{nu3} hold with $\frac{s+2}{s-2}\le\alpha\le\frac{2s}{s-2}$. In addition, assume
\begin{equation}\label{epsilon10}
0<\epsilon_1\leq\min\left(\frac{1}{s^2}, \frac{s-2}{2s^2+3s+2}\right), ~~\mbox{fixed}.
\end{equation}
For every $(u_0,u_1)\in \mathcal{H}_{1} \times \mathcal{H}$, assume $t_5\ge t_1$ with $t_1$ in \eqref{t01} and
\begin{equation}\label{t04}
\begin{aligned}
t_5\geq&
\frac{\nu^2+2\nu+3}{\nu \epsilon_1}\alpha+
\left[\frac{4\bar{C}_1}{\epsilon_1\alpha}\left(\left[H^*(t_0)\right]^{\frac{s-2}{s}}+\left[H^*(t_0)\right]^{\frac{s-2}{2}}\right)\right]
^{\frac{s}{2}}\\
&+\left(\frac{8\bar{C}_1}{\epsilon_1\alpha}\right)^{s}\left[\tilde H^*(t_0)+C_\varepsilon t^{\alpha-\frac{2s}{s-2}-2\varepsilon}_{0} \right]^{s-2},
\end{aligned}
\end{equation}
where $\nu$ is from \eqref{nu3}, $\varepsilon$ is from \eqref{vare}, $H^*(t_0)$, $\tilde H^{*}(t_0)$ and $C_\varepsilon$ are defined in Lemma \ref{th2'}, and
\begin{equation}\label{C1}
\bar{C}_1=2^{\frac{s}{2}}R\left(M^{\frac{s-2}{s}}_{0}+1\right),
\end{equation}
with $M_0$ and $R$ defined in H3. Let $v=Qu$ solve \eqref{vt}.

Then, $\forall t\ge t_5$,
$$
\|v(t)\|^2+\|v'(t)\|^2+\|A^{1/2}v(t)\|^2\leq \frac{1+\epsilon_1}{1-\epsilon_1}\cdot\frac{2\nu+2}{\nu}E_v(t_5)t^{\frac{1-\epsilon_1}{1+\epsilon_1}\alpha}_{5}
t^{-\frac{1-\epsilon_1}{1+\epsilon_1}\alpha},
$$
where $E_v(t)$ is defined in \eqref{venergy}.
\end{lemma}

Letting $\epsilon_1<\frac{1}{3}$, we infer
\begin{equation}\label{vv}
\|v(t)\|^2+\|v'(t)\|^2+\|A^{1/2}v(t)\|^2\leq \frac{4\nu+4}{\nu}E_v(t_5)t^{\frac{1-\epsilon_1}{1+\epsilon_1}\alpha}_{5}
t^{-\frac{1-\epsilon_1}{1+\epsilon_1}\alpha}.
\end{equation}
As a result, by \eqref{t04}, we have
\begin{equation}\label{lambdav+tv'}
\begin{aligned}
\alpha\| v(t)\|+t\|v'(t)\|&\leq (\alpha+t)\left[\frac{4\nu+4}{\nu}E_v(t_5)t^{\frac{1-\epsilon_1}{1+\epsilon_1}\alpha}_{5}
t^{-\frac{1-\epsilon_1}{1+\epsilon_1}\alpha}\right]^{\frac{1}{2}}\\
&\leq 2t\left[\frac{4\nu+4}{\nu}E_v(t_5)t^{\frac{1-\epsilon_1}{1+\epsilon_1}\alpha}_{5}
t^{-\frac{1-\epsilon_1}{1+\epsilon_1}\alpha}\right]^{\frac{1}{2}}.
\end{aligned}
\end{equation}

Making use of the Lemma \ref{lem8}, we can obtain a result of lower decay estimate, as follows.

\begin{theorem}\label{thm3}
Assume (H1), (H2), (H3), \eqref{nu3} and \eqref{vare} hold with $A\neq 0.$  Let \eqref{gamma} hold with $\frac{s+2}{s-2}\leq\alpha\le2$.
For every $(u_0,u_1)\in \mathcal{H}_{1} \times \mathcal{H}\setminus\left(\ker A \times \ker A\right)$, assume $t_6$ is large enough such that \eqref{t04} and
\begin{equation}\label{t05}
t_6\ge\left[2^{\frac{s}{s-2}}H^*(t_0)\right]^\frac{s(s+2)}{2\alpha}+\left(8s\bar{C}_3\right)^s+\frac{4\sqrt{2\nu+2}}{\nu},
\end{equation}
are satisfied, where
\begin{equation}\label{C3}
\bar{C}_3=\frac{8\nu+8}{\nu}\bar{C}_1,
\end{equation}
with $\bar{C}_1$ defined in \eqref{C1}, and $H^*(t_0)$ stands for $H(t_0)$ with $\eta=\alpha-2$ and $\lambda=\frac{\alpha}{2}$.

Then, the solution energy of \eqref{1.1} satisfies
\begin{equation}\label{Elower'}
E(t)\ge ct^{-\alpha}, \quad\forall t\ge t_6.
\end{equation}
\end{theorem}

\begin{proof}
Choosing $T=t_6$ in \eqref{tvare1}, we need to estimate
$$
\int^{t}_{t_6}\tau^{\eta+1}\|\lambda v(\tau)+\tau v'(\tau)\|\|Q\nabla F(u(\tau))\|d\tau,
$$
and
$$
\int^{t}_{t_6}\frac{\lambda(\alpha-2)}{2}||v(\tau)||^2\tau^{\eta-1}d \tau.
$$

First, using (H3)(1), (H3)(2) and \eqref{ualpha<2}, we have
\begin{equation}\label{QF1}
\begin{aligned}
\|Q\nabla F(u(t))\|
\leq & R\left(M^{\frac{s-2}{s}}_{0}\left(\|A^{1/2}u\|^{2}+F(u)\right)^{\frac{s-2}{s}}+\|A^{1 / 2} u(t)\|^{s-2}\right)\|A^{1 / 2} v(t)\|\\
\leq& \bar{C}_1
\left(\left[H^*(t_0)\right]^{\frac{s-2}{s}}+\left[H^*(t_0)\right]^{\frac{s-2}{2}}\right)t^{-\frac{s-2}{s}\alpha}\|A^{1 / 2} v(t)\|,
\end{aligned}
\end{equation}
where $\bar{C}_1$ is a constant defined in \eqref{C1}. Combining this inequality, $\lambda=\frac{\alpha}{2}$, \eqref{vv}, \eqref{lambdav+tv'} and \eqref{C3}, we obtain
\begin{equation}\label{tau vv'}
\begin{aligned}
&\int^{t}_{t_6}\tau^{\eta+1}\|\lambda v(\tau)+\tau v'(\tau)\|\|Q\nabla F(u(\tau))\|d\tau\\
\leq& \bar{C}_1\left(\left[H^*(t_0)\right]^{\frac{s-2}{s}}+\left[H^*(t_0)\right]^{\frac{s-2}{2}}\right)\int^{t}_{t_6}\left(\lambda \|v(\tau)\|+\tau\| v'(\tau)\|\right)\|A^{1 / 2} v(t)\|\tau^{\eta+1-\frac{s-2}{s}\alpha}d\tau\\
\leq& \bar C_3\left(\left[H^*(t_0)\right]^{\frac{s-2}{s}}+\left[H^*(t_0)\right]^{\frac{s-2}{2}}\right)
E_v(t_6)t^{\frac{1-\epsilon_1}{1+\epsilon_1}\alpha}_{6}
\int^{t}_{t_1}\tau^{\frac{2\epsilon_1}{1+\epsilon_1}\alpha-\frac{s-2}{s}\alpha}d\tau\\
\leq& s\bar C_3\left(\left[H^*(t_0)\right]^{\frac{s-2}{s}}+\left[H^*(t_0)\right]^{\frac{s-2}{2}}\right)E_v(t_6)t^{1+\frac{2}{s}\alpha}_{6},
\end{aligned}
\end{equation}
due to \eqref{epsilon10} and
$$
\frac{2\epsilon_1}{1+\epsilon_1}\alpha-\frac{s-2}{s}\alpha+1<-\frac{1}{s}.
$$

Then, using \eqref{t05}, we find
\begin{align*}
\left[H^*(t_0)\right]^{\frac{s-2}{s}}+\left[H^*(t_0)\right]^{\frac{s-2}{2}}
\leq t_6^{\frac{s-2}{s(s+2)}\alpha},
\end{align*}
which implies that \eqref{tau vv'} becomes
\begin{equation}\label{1/s}
\begin{aligned}
\int^{t}_{t_6}\tau^{\eta+1}\|\lambda v(\tau)+\tau v'(\tau)\|\|Q\nabla F(u(\tau))\|d\tau\leq s\bar C_3E_v(t_6)t_6^{1+\frac{3s+2}{s(s+2)}\alpha}
\leq s\bar C_3E_v(t_6)t_6^{\alpha-\frac{1}{s}},
\end{aligned}
\end{equation}
where the last inequality is from
\begin{align*}
\alpha-\frac{3s+2}{s(s+2)}\alpha-\frac{1+s}{s}=\frac{s^2-s-2}{s^2+2s}\alpha-\frac{1+s}{s}\ge0,~\mbox{by}~\alpha\ge \frac{s+2}{s-2}.
\end{align*}

On the other hand, using \eqref{nu3} and \eqref{vv}, with $\bar{C}_4=\frac{4\nu+4}{\nu^2}$, we have
\begin{align*}
\int^{t}_{t_6}||v(\tau)||^2\tau^{\eta-1}d \tau&\le \frac{1}{\nu} \int^{t}_{t_6}||A^{1/2}v(\tau)||^2\tau^{\eta-1}d \tau\\
&\le \bar{C}_4E_v(t_6)t^{\frac{1-\epsilon_1}{1+\epsilon_1}\alpha}_{6}\int^{t}_{t_6}\tau^{\frac{2\epsilon_1}{1+\epsilon_1}\alpha-3}d \tau\\
&\le \bar{C}_4E_v(t_6)t^{\alpha-2}_{6},
\end{align*}
where the last inequality is from
$
\frac{2\epsilon_1}{1+\epsilon_1}\alpha-2<-1
$
(by \eqref{epsilon10}).
From these two estimates, $\eta=\alpha-2$, \eqref{tvare1}, \eqref{t05} and Lemma \ref{lem4}, it follows that
\begin{align*}
t^{\eta}E_{\lambda,v}(t)\geq t^{\eta}_{6}E_{\lambda,v}(t_6)-s\bar C_3E_v(t_6)t_{6}^{\alpha-\frac{1}{s}}-\bar C_4E_v(t_6)t^{\alpha-2}_{6}
\ge \frac{1}{4}t^{\alpha}_{6}E_{v}(t_6).
\end{align*}
Again, using Lemma \ref{lem4}, we obtain
$
E(t)\ge E_{v}(t)\ge 1/8t^{\alpha}_{6}E_{v}(t_6)t^{-\alpha}.
$
Hence, to obtain \eqref{Elower'}, the initial data should satisfy
$$(u_0,u_1)\in \mathcal{H}_{1} \times \mathcal{H}\setminus\left(\ker A \times \ker A\right),$$
such that $E_{v}(t_6)\neq 0$, and we complete the proof.
\end{proof}

\subsection{$\frac{s+2}{s-2}\leq2<\alpha\le\frac{2s}{s-2}$  or $2<\frac{s+2}{s-2}\leq\alpha\le\frac{2s}{s-2}$ }

The situation is more challenging, since \eqref{ualpha<2} is not satisfied. However, we can still obtain the existence of slow solutions in an open set. The second theorem (in this section) of lower decay estimate is stated as follows.
\begin{theorem}\label{thm4}
Assume (H1), (H2), (H3), \eqref{gamma}, and \eqref{nu3} with $A\neq0$. Assume either $\alpha\in \left(2, \frac{2s}{s-2}\right]$ with $s>6$, or $\alpha\in \left(\frac{s+2}{s-2},\frac{2s}{s-2}\right]$ with $s\le6$.
For every $(u_0,u_1)\in \mathcal{H}_{1} \times \mathcal{H}\setminus\left(\ker A \times \ker A\right)$, assume that $t_7$ is large enough such that \eqref{t01}, \eqref{t04}, and
\begin{equation}\label{t06}
t_7\ge \left[2\tilde{H}^*(t_0)\right]^{\frac{s+2}{\alpha}}+\left[2C_\varepsilon t^{\alpha-\frac{2s}{s-2}-2\varepsilon}_{0}\right]^{\frac{s(s+2)}{2\alpha}}+\left(16s\bar{C}_3\right)^{2s}+\frac{4\sqrt{2\nu+2}}{\nu},
\end{equation}
are satisfied, where $\varepsilon$ is from \eqref{vare}, $\tilde{H}^*(t_0)$ and $C_\varepsilon $ are defined in \eqref{ualpha>2}, and $\bar{C}_3$ is in \eqref{C3}.

Then, the solution energy of \eqref{1.1} satisfies
$$
E(t)\ge c t^{-\alpha},\quad t\ge t_7.
$$
\end{theorem}

\begin{proof}
Let $\eta=\alpha-2$ and $\lambda=\frac{\alpha}{2}$ as in Theorem \ref{thm3}. It follows from \eqref{tvare1} with $T=t_7$ and $\alpha\ge2$ that
\begin{align*}
t^{\eta}E_{\lambda,v}(t)
\ge &t^{\eta}_{7}E_{\lambda,v}(t_7)-\int^{t}_{t_7}\left\langle\lambda v(\tau)+\tau v'(\tau), \tau^{\eta+1} Q\nabla F(u(\tau))\right\rangle d \tau\\
\ge &t^{\eta}_{7}E_{\lambda,v}(t_7)-\int^{t}_{t_7}\tau^{\eta+1}\|\lambda v(\tau)+\tau v'(\tau)\|\|Q\nabla F(u(\tau))\|d\tau.
\end{align*}
From Lemma \ref{lem8}, we find that \eqref{vv} and \eqref{lambdav+tv'} still hold in this case.

Then, using (H3)(2), \eqref{ualpha>2} and \eqref{vv} we have, for any $t\ge t_7$,
\begin{equation}\label{QF1'}
\begin{aligned}
&\|Q\nabla F(u(t))\|\\
\leq& R\left(M^{\frac{s-2}{s}}_{0}\left(\|A^{1/2}u\|^{2}+F(u)\right)^{\frac{s-2}{s}}+\|A^{1 / 2} u(t)\|^{s-2}\right)\|A^{1 / 2} v(t)\|\\
\leq& R\left(M^{\frac{s-2}{s}}_{0}+1\right)\left[\tilde H^*(t_0)t^{-\alpha+2\varepsilon}+C_\varepsilon t^{\alpha-\frac{2s}{s-2}-2\varepsilon}_{0} t^{-\alpha+2\varepsilon}\right]^{\frac{s-2}{s}}\|A^{1 / 2} v(t)\|\\
\leq& R\left(M^{\frac{s-2}{s}}_{0}+1\right)\left[\tilde H^*(t_0)+C_\varepsilon t^{\alpha-\frac{2s}{s-2}-2\varepsilon}_{0} \right]^{\frac{s-2}{s}}t^{(-\alpha+2\varepsilon)\cdot\frac{s-2}{s}}\|A^{1 / 2} v(t)\|.
\end{aligned}
\end{equation}
Combining \eqref{vare} and \eqref{t06}, we deduce
$$
\tilde H^*(t_0)\leq\frac{1}{2}t^{\frac{\alpha}{s+2}}_{7},
$$
and
$$
C_\varepsilon t^{\alpha-\frac{2s}{s-2}-2\varepsilon}_{0}\leq\frac{1}{2}t^{\frac{2\alpha}{s(s+2)}}_{7},
$$
which indicates that \eqref{QF1'} becomes
\begin{equation}\label{QF2}
\begin{aligned}
\|Q\nabla F(u(t))\|\leq& R\left(M^{\frac{s-2}{s}}_{0}+1\right)\left[\frac{1}{2}t^{\frac{\alpha}{s+2}}_{7}+\frac{1}{2}t^{\frac{2\alpha}{s(s+2)}}_{7}
\right]^{\frac{s-2}{s}}\|A^{1 / 2} v(t)\|\\
\leq & \bar{C}_1t^{\frac{s-2}{s(s+2)}\alpha}_{7}t^{\left(-\alpha+2\varepsilon\right)\cdot\frac{s-2}{s}}\|A^{1 / 2} v(t)\|,
\end{aligned}
\end{equation}
by $s>2$, \eqref{C1} and \eqref{t06}. Using the inequality and \eqref{lambdav+tv'}, we get
\begin{align*}
&\int^{t}_{t_7}\tau^{\eta+1}\|\lambda v(\tau)+\tau v'(\tau)\|\|Q\nabla F(u(\tau))\|d\tau\\
\leq& \bar C_3E_v(t_7)t^{\frac{1-\epsilon_1}{1+\epsilon_1}\alpha+\frac{s-2}{s(s+2)}\alpha}_{7}
\int^{t}_{t_7}\tau^{\frac{2\epsilon_1}{1+\epsilon_1}\alpha-\frac{s-2}{s}(\alpha-2\varepsilon)}d\tau\\
\leq& 2s\bar C_3E_v(t_7)t^{1+\frac{2}{s}\alpha+2\varepsilon\frac{s-2}{s}}_{7},
\end{align*}
due to
$$
\frac{2\epsilon_1}{1+\epsilon_1}\alpha-\frac{s-2}{s}(\alpha-2\varepsilon)+1<-\frac{1}{2s},
$$
by \eqref{vare} and \eqref{epsilon10}.
As a result, using \eqref{vare} again and \eqref{1/s}, we obtain
\begin{align*}
\int^{t}_{t_7}\tau^{\eta+1}\|\lambda v(\tau)+\tau v'(\tau)\|\|Q\nabla F(u(\tau))\|d\tau
\leq 2s\bar C_3E_v(t_7)t_7^{\alpha-\frac{1}{2s}}.
\end{align*}
Then, proceeding similarly as in the case when $\alpha\le2$, we verify the conclusion. 
\end{proof}

\subsection{Remarks and examples}

To sum up, as a consequence of Theorems \ref{thm3} and \ref{thm4}, the allowable value range of $\alpha$
 is exactly $\left[\frac{s+2}{s-2},~\frac{2s}{s-2}\right]$.
Hence, the optimal energy decay rate is $t^{-\alpha}$ for these $\alpha$ (except for $\alpha=\frac{2s}{s-2}$) in view of \eqref{nA1}, which is determined by the range component $v$ of the solution $u$. Also we recall that the case of $\alpha>2s/(s-2)$ has been treated in \cite{LUO3}, in which the optimal energy decay rate is determined by the kernel component $w$.

When $F$ is absent, \eqref{1.1} becomes a linear equation. In addition, \eqref{wt} and \eqref{vt} become
\begin{eqnarray}\label{wt1}
\left\{\begin{array}{ll}
w''(t)+\gamma(t)w'(t)=0, \quad \forall t\geq t_0,\\[0.36cm]
w(t_0)=P u_{0}, ~w'(t_0)=P u_{1},
\end{array} \right.
\end{eqnarray}
 and
\begin{eqnarray}\label{vt1}
\left\{\begin{array}{ll}
v''(t)+A v(t)+\gamma(t)v'(t)=0, \quad \forall t\geq t_0,\\[0.36cm]
v(t_0)=Q u_{0}, ~v'(t_0)=Q u_{1},
\end{array} \right.
\end{eqnarray}
with $u=v+w=Qu+Pu$. Hence
\begin{align*}
E(u(t))&=\frac{1}{2}\|u'(t)\|^2+\frac{1}{2}\|A^{1/2}u(t)\|^2\\
&=\frac{1}{2}\|v'(t)\|^2+\frac{1}{2}\|A^{1/2}v(t)\|^2+\frac{1}{2}\|w'(t)\|^2\\
&=E_v(t)+E_w(t).
\end{align*}
The following result has no restriction of $\alpha$, even if $A$ is not coercive.
\begin{corollary}\label{cor}
Suppose that Assumptions (H1), (H2)(1), \eqref{nu3} and \eqref{gamma} hold. Let $(u_0, u_1)\in \mathcal{H}_1\times \mathcal{H}$, and $u(t)$ be the solution of \eqref{1.1} with $F\equiv 0$. Then the energy of $u(t)$ satisfies
$$
E(t)\leq CE(t_0)t^{-\alpha},\quad \forall t\ge t_{3,\varepsilon}, ~~ \mbox{with some positive constant}~C,
$$
where $t_{3,\varepsilon}$ is defined in \eqref{t03} satisfying \eqref{t01} and \eqref{t02}.
Moreover, for every $(u_0,u_1)\in \mathcal{H}_{1} \times \mathcal{H}\setminus\{(0,0)\}$, we have
$$
E(t)\ge cE(t_0)t^{-\alpha},\quad \forall t\ge t_{3,\varepsilon}, ~~ \mbox{with some positive constant}~c.
$$
\end{corollary}

\begin{proof}
First, from \eqref{wt1}, it is a simple calculation that
$$
w'(t)=w'(t_0)\left(\frac{t_0}{t}\right)^{\alpha},
$$
which infers that
\begin{equation}\label{Ew}
E_{w}(t)=\frac{1}{2}\|w'(t)\|^{2}=\frac{t^{2\alpha}_{0}}{2}\|w'(t_0)\|^{2}t^{-2\alpha}.
\end{equation}
As for $E_{v}(t)$, from \eqref{4Ht0} and \eqref{Eut3}, we have
$$
\frac{t^{-\alpha}_{3,\varepsilon}}{4}t^{2\alpha}_{0}E_v(t_0) t^{-\alpha}\leq E_v(t)\leq 8t^{\alpha}_{3,\varepsilon}E(t_0)t^{-\alpha},\quad\forall t\ge t_{3,\varepsilon}.
$$
Hence, combining the above inequality, $E(t)=E_{w}(t)+E_{v}(t)$ and \eqref{Ew}, we immediately obtain the results.
\end{proof}

Next, we give two examples indicating the applicability of our abstract theorems.

\vspace{4pt}
\noindent
\textbf{Example 1.} We consider the wave equation with Neumann boundary condition:
\begin{eqnarray}\label{6.2}
\left\{\begin{array}{ll}
&u_{tt}(t,x)-\Delta u(t,x)+\gamma(t) u_{t}(t,x)+|u(t,x)|^{p}u(t,x)=0, \quad \mbox{in}~ (0,+\infty)\times\Omega,\\[0.2cm]
&u(0,x)=u_{0}(x),~ u_{t}(0,x)=u_{1}(x), \quad x\in\Omega ,\\[0.36cm]
&\displaystyle\frac{\partial u(t,x)}{\partial \nu}=0, \quad \mbox{on}~ \partial \Omega \times  (0,\infty),
\end{array} \right.
\end{eqnarray}
where $\Omega$ is a bounded domain in $R^{n}$ ($n$ a natural number), with smooth boundary $\partial \Omega$, and $\nu$ the unit outward normal on $\partial\Omega$. The $p$ is a positive exponent satisfying $(n-2)p \leq 2 $.

Take $\mathcal{H}=L^{2}(\Omega)$, and $A=-\Delta_{N}$ (the negative Neumann-Laplacian on $\Omega$), and let
$$F(u)=\frac{1}{p+2}\int_{\Omega}|u(x)|^{p+2}dx,\quad \forall u\in \mathcal{H}_1=H^{1}(\Omega).$$
Then \eqref{6.2} can be rewritten in the form of \eqref{1.1}. It can be verified (see \cite[Section 7]{LUO3} for details) that (H1)-(H3), and \eqref{nu3} hold. Therefore,
Theorem \ref{thm3} and \ref{thm4} are applicable to \eqref{6.2}.

\vspace{4pt}
\noindent
\textbf{Example 2.} Consider the initial-boundary value problem for an integro-differential damped hyperbolic equation
\begin{eqnarray}\label{6.3}
\left\{\begin{array}{ll}
&u_{tt}(t,x)-\Delta u(t,x)-\lambda_1u(t, x)+\gamma(t)u_t\\[0.36cm]
&~~+\left(\displaystyle \int_{\Omega}|u(t,x)|^{2}dx\right)^{p/2}u(t,x)=0,
 ~\mbox{in}~ [0,+\infty)\times\Omega,\\[0.7cm]
&u(0,x)=u_{0}(x),~~ u_{t}(0,x)=u_{1}(x), \quad x\in\Omega ,\\[0.36cm]
& u(t,x)=0, \quad \mbox{on}~\partial \Omega \times  (0,\infty),
\end{array} \right.
\end{eqnarray}
where $\Omega$ is as in Example 1, $p\ge 1,$ and $\lambda_1$ is the first eigenvalue of the negative Dirichlet-Laplacian $-\Delta_D$ on $\Omega$.

Take
$\mathcal{H}=L^{2}(\Omega)$, $A=-\lambda_1I-\Delta_D$, and let
\begin{align}\label{F2}
F(u)=\frac{1}{p+2}\left(\int_{\Omega}|u(x)|^{2}dx\right)^{\frac{p+2}{2}}, \quad u\in \mathcal{H}_1=H_0^{1}(\Omega).
\end{align}
Then (H1)-(H3), and \eqref{nu3} are satisfied (see also \cite[ Section 7]{LUO3}). Thus, Theorems \ref{thm3} and \ref{thm4} are applicable to \eqref{6.3}, as well.

We note that the condition (H3)(2) is not suitable for the case when $A =-\lambda_1I-\Delta_D$, and $F(u)=\frac{1}{p+2}\int_{\Omega}|u(x)|^{p+2}dx$ (as in Example 1).  The problem whether condition (H3)(2) in Theorems \ref{thm3} and \ref{thm4} can be dropped remains open.

\section{Proofs of Auxiliary Lemmas }

\noindent {\bf Proof of Lemma \ref{th2'} .}
For $\alpha\le2$, choose $\eta=\alpha-2, \lambda=\frac{\alpha}{2}$, and so
\begin{equation}\label{xi1}
\xi=\lambda(\lambda+1-\alpha)=\frac{\alpha}{4}(2-\alpha)\ge0.
\end{equation}
We see that \eqref{H'alpha} is still satisfied. This and $\alpha\le2$ together yield that $\tilde H'(t)\leq0$. Thus,
$$
\tilde E_{\lambda}(t)=\tilde H(t)t^{-\eta}\leq \tilde H(t_0) t^{-\eta}.
$$
Combining $\alpha\le2$ and \eqref{xi1} gives
$$
t^{2}F(u(t))+\frac{1}{2}\left\|\lambda u(t)+t u'(t)\right\|^{2}+\frac{t^2}{2}\|A^{1/2}u(t)\|^2\leq\tilde H(t_0) t^{-\eta},
$$
and so \eqref{ualpha<2} holds.

For $\alpha>2$,  choose $\eta=\alpha-2-2\varepsilon$ and $\lambda=\frac{\alpha}{2}+\varepsilon$, where $\varepsilon$ is a constant that can be arbitrarily small, we find that \eqref{H'a>2} still hold. That is
\begin{equation}\label{H'2}
\begin{aligned}
\frac{d}{dt}\tilde{H}^{*}(t)
\leq& t^{\eta}\left[\frac{(2-s)\alpha}{2} a_u(t)+\lambda(\alpha-2-2\varepsilon)(1+2\varepsilon)c_u(t)-4\varepsilon d_u(t)\right]\\
\leq& t^{\eta}\left[\frac{(2-s)\alpha}{2} tF(u(t))+\frac{\alpha+2\varepsilon}{4t}|\alpha-2-2\varepsilon|(1+2\varepsilon)\|u(t)\|^2-2\varepsilon t\|A^{1/2}u(t)\|^2\right].
\end{aligned}
\end{equation}
Using Young's inequality with two exponents $\frac{s}{2}$ and $\frac{s}{s-2}$, we have
$$
t^{-1}\|u(t)\|^2=(t^{-1-\frac{2}{s}})(t^{\frac{2}{s}}\|u(t)\|^{2})\leq \varepsilon_0t\|u(t)\|^s+C_{\varepsilon_0}t^{-\frac{s+2}{s-2}},
$$
where $0<\varepsilon_0\ll \varepsilon$ is a sufficiently small constant such that
\begin{equation}\label{ep0}
\begin{aligned}
\frac{\alpha+2\varepsilon}{4t}|\alpha-2-2\varepsilon|(1+2\varepsilon)\|u(t)\|^2&\leq \alpha|\alpha-2-2\varepsilon|\left(\varepsilon_0t\|u(t)\|^s+C_{\varepsilon_0}t^{-\frac{s+2}{s-2}}\right)\\
&\leq \varepsilon t(F(u)+\|A^{1/2}u(t)\|^2)+C_{\varepsilon_0}t^{-\frac{s+2}{s-2}},
\end{aligned}
\end{equation}
by (H3)(1).
Combining \eqref{ep0}, $\eta=\alpha-2-2\varepsilon$ and \eqref{H'2}, we deduce
\begin{align*}
\frac{d}{dt}\tilde{H}^{*}(t)\leq C_{\varepsilon_0}t^{\eta-\frac{s+2}{s-2}}=C_{\varepsilon_0}t^{\alpha-\frac{3s-2}{s-2}-2\varepsilon},
\end{align*}
and hence
\begin{align*}
\tilde{H}^{*}(t)\leq \tilde{H}^{*}(t_0)+\int^{t}_{t_0}C_{\varepsilon_0}\tau^{\alpha-\frac{3s-2}{s-2}-2\varepsilon}d\tau
\leq \tilde{H}^{*}(t_0)-\frac{C_{\varepsilon_0}}{\alpha-\frac{2s}{s-2}-2\varepsilon}t^{\alpha-\frac{2s}{s-2}-2\varepsilon}_{0},
\end{align*}
by $\alpha\le\frac{2s}{s-2}$. Then, it follows from $\xi<0$ and \eqref{ep0} that
\begin{align*}
&t^{\eta+2}\left(F(u(t))+\frac{1}{2}\|A^{1/2}u(t)\|^2\right)\\
\leq& \tilde{H}^{*}(t_0)+\frac{\lambda|\lambda +1-\alpha|}{2}t^{\eta}\|u(t)\|^2-
\frac{C_{\varepsilon_0}}{\alpha-\frac{2s}{s-2}-2\varepsilon}t^{\alpha-\frac{2s}{s-2}-2\varepsilon}_{0}\\
\leq& \tilde{H}^{*}(t_0)+\varepsilon t^{\alpha-2\varepsilon}(F(u)+\|A^{1/2}u(t)\|^2)+C_{\varepsilon_0}t^{\alpha-\frac{2s}{s-2}-2\varepsilon}_{0},
\end{align*}
due to $\alpha\le\frac{2s}{s-2}$. From \eqref{vare}, for any $t\ge t_0$, we obtain
$$
\frac{1}{2}t^{\eta+2}\left(F(u(t))+\frac{1}{2}\|A^{1/2}u(t)\|^2\right)\leq \tilde{H}^{*}(t_0)+C_{\varepsilon_0}t^{\alpha-\frac{2s}{s-2}-2\varepsilon}_{0}.
$$
This and $\eta=\alpha-2-2\varepsilon$ justify \eqref{ualpha>2}. Thus we complete the proof.

\noindent {\bf Proof of Lemma \ref{lem8}.}
We consider the energies
$$E_v(t)=\frac{1}{2}\|v'(t)\|^2+\frac{1}{2}\|A^{1/2}v(t)\|^2,~~
\hat{E}_v(t)=E_v(t)+\gamma(t)\langle v(t),v'(t)\rangle.$$
From \eqref{vnu} and \eqref{t04}, it follows that
\begin{align*}
\gamma(t)\langle v(t),v'(t)\rangle\leq\gamma(t)\frac{\|v'(t)\|}{\sqrt{\nu}}\cdot\sqrt{\nu}\|v(t)\|
\leq\frac{\epsilon_1}{2}\left(\|v'(t)\|^2+\|A^{1/2}v(t)\|^2\right),
\end{align*}
which implies
\begin{equation}\label{EvhatE}
(1-\epsilon_1)E_v(t)\leq\hat{E}_v(t)\leq(1+\epsilon_1)E_v(t).
\end{equation}
Also,
\begin{align*}
\hat{E}'_v(t)=&-\gamma(t)\|v'(t)\|^2-\gamma(t)\|A^{1/2}v(t)\|^2-\gamma^2(t)\langle v(t),v'(t)\rangle\\
&+\gamma'(t)\langle v(t),v'(t)\rangle-2\langle Q\nabla F(u),v'(t)\rangle-\gamma(t)\langle Q\nabla F(u),v(t)\rangle\\
=&k_1+k_2+k_3+k_4+k_5+k_6.
\end{align*}
Let us estimate separately the sum $k_{3}+k_{4}$ and the last two terms. First, \eqref{gamma} implies
$$
|\gamma'(t)|=\frac{1}{\alpha}\gamma^2(t)\leq\gamma^2(t),
$$
by $\alpha\ge \frac{s+2}{s-2}>1.$ Then, from \eqref{t04}, we obtain
\begin{equation}\label{k34}
\begin{aligned}
k_3+k_4\leq&\gamma^2(t)\| v(t)\|\cdot\|v'(t)\|+|\gamma'(t)|\| v(t)\|\cdot\|v'(t)\|\\
\leq&\frac{\epsilon_1}{2}\gamma(t)\|A^{1/2}v(t)\|^2+\frac{\epsilon_1}{2}\gamma(t)\|v'(t)\|^2.
\end{aligned}
\end{equation}

Next, we consider separately two cases.

\vspace{0.3cm}

\noindent{\bf Case I}: {\it $\frac{s+2}{s-2}\le\alpha<2$}.
In this case, \eqref{QF1} is satisfied. Then, using \eqref{t04} and $\alpha\ge\frac{s+2}{s-2}$, we have
\begin{equation}\label{QFeps4'}
\begin{aligned}
\|Q\nabla F(u)\| \leq \bar C_1
\left(\left[H^*(t_0)\right]^{\frac{s-2}{s}}+\left[H^*(t_0)\right]^{\frac{s-2}{2}}\right)t^{-\frac{s-2}{s}\alpha}\|A^{1 / 2} v(t)\|
\leq\frac{\epsilon_1}{4}\gamma(t)\|A^{1 / 2} v(t)\|.
\end{aligned}
\end{equation}

\vspace{0.3cm}

\noindent{\bf Case II}: {\it $\alpha\ge \frac{s+2}{s-2}>2$ or $\alpha\ge2> \frac{s+2}{s-2}$}.
In this case, \eqref{QF2} is satisfied. Also, using \eqref{vare}, \eqref{t04} and $\alpha\ge\frac{s+2}{s-2}$, we have
\begin{equation}\label{QFeps4''}
\begin{aligned}
\|Q\nabla F(u)\| \leq \bar C_1 \left[\tilde H^*(t_0)+C_\varepsilon t^{\alpha-\frac{2s}{s-2}-2\varepsilon}_{0} \right]^{\frac{s-2}{s}}t^{\left(-\alpha+2\varepsilon\right)\cdot\frac{s-2}{s}}\|A^{1 / 2} v(t)\|
\leq\frac{\epsilon_1}{4}\gamma(t)\|A^{1 / 2} v(t)\|.
\end{aligned}
\end{equation}
According to \eqref{QFeps4'} and \eqref{QFeps4''}, we infer
\begin{equation}\label{k5}
\begin{aligned}
2\langle Q\nabla F(u),v'(t)\rangle
\leq\frac{\epsilon_1}{4}\gamma(t)\|A^{1/2}v(t)\|^2+\frac{\epsilon_1}{2}\gamma(t)\|v'(t)\|^2,
\end{aligned}
\end{equation}
and
\begin{align*}
\langle Q\nabla F(u),v(t)\rangle\leq \frac{\|Q\nabla F(u)\|}{\sqrt{\nu}}\cdot\sqrt{\nu}\|v(t)\|\leq\frac{\epsilon_1}{4}\gamma(t)\|A^{1/2}v(t)\|^2.
\end{align*}
Combining this, \eqref{k34} and \eqref{k5} we obtain, for any $t\ge t_5$,
$$
\hat{E}'_v(t)\leq-(1-\epsilon_1)\gamma(t)(\|v'(t)\|^2+\|A^{1/2}v(t)\|^2)\leq-\frac{1-\epsilon_1}{1+\epsilon_1}\gamma(t)\hat{E}_v(t),
$$
from which it follows that
$$
\hat{E}_v(t)\leq \hat{E}_v(t_5)\left(\frac{t}{t_5}\right)^{-\frac{1-\epsilon_1}{1+\epsilon_1}\alpha}.
$$
Therefore, using \eqref{vnu} and \eqref{EvhatE} yields
\begin{align*}
\|v(t)\|^2+\|v'(t)\|^2+\|A^{1/2}v(t)\|^2\leq \frac{1+\epsilon_1}{1-\epsilon_1}\cdot\frac{2\nu+2}{\nu}E_v(t_5)t^{\frac{1-\epsilon_1}{1+\epsilon_1}\alpha}_{5}
t^{-\frac{1-\epsilon_1}{1+\epsilon_1}\alpha}.
\end{align*}

\vspace{0.5cm}

{\bf Contributions}

The authors contributed equally and significantly in writing this paper.

{\bf Data Availability}

Data sharing is not applicable to this article as no datasets were generated or analyzed during the current study.

{\bf Conflict of interest}

The authors have no competing interests to declare.

{\bf Ethical Statement}

There are no ethical concerns applicable to our research.

\end{document}